\documentclass[a4paper, 11pt,reqno]{amsart}

\usepackage{amsmath,amsthm,amsbsy,amssymb}
\usepackage{arydshln} 
\usepackage{booktabs} 
\usepackage{lipsum}
\usepackage{comment}

\makeatletter
\newcommand{\leqnos}{\tagsleft@true\let\veqno\@@leqno}
\newcommand{\reqnos}{\tagsleft@false\let\veqno\@@eqno}
\reqnos
\makeatother

\usepackage{xcolor} 
\definecolor{orange}{rgb}{1,0.5,0}
\definecolor{Red}{rgb}{.795,0.015,0.017}
\definecolor{Ggreen}{rgb}{0.,0.675,0.0128}
\definecolor{Bblue}{rgb}{0.16,.32,0.91}
 \usepackage[backref=page]{hyperref}
\hypersetup{
    colorlinks = true,
linkcolor={red},
urlcolor={blue},
citecolor={Ggreen},    
urlcolor = {blue},
citebordercolor = {0.33 .58 0.33},
 linkbordercolor = {0.99 .28 0.23},
 pdftitle={Statistical Distribution of Fermat Quotients},
 pdfauthor={ACVZ}
}

\usepackage{mathrsfs}

\def\cC{\mathcal C}

\newcommand{\cT}{\mathcal T}

\def\cS{\mathcal S}

\usepackage{bm}
\newcommand*{\B}[1]{\ifmmode\bm{#1}\else\textbf{#1}\fi}

\def\AA{\mathbb{A}}
\def\CC{\mathbb{C}}
\def\FF{\mathbb{F}}

\def\NN{\mathbb{N}}

\renewcommand{\NN}{\mathbb{Z}_{\ge 0}}

\def\ZZ{\mathbb{Z}}
\def\QQ{\mathbb{Q}}

\newcommand{\sdfrac}[2]{\mbox{\small$\displaystyle\frac{#1}{#2}$}}

\usepackage{xspace}

\newcommand{\balpha}{\boldsymbol{{\alpha\xspace}}} 

\DeclareMathOperator{\wt}{\mathtt{Hwt}} 
\newcommand{\Crec}{$\mathrm{C}$-recursive\xspace} 

\newcommand{\bA}{\textbf{A}}
\DeclareMathOperator{\Auto}{\bA} 

\newcommand*\dd{\mathop{}\!\mathrm{d}}

\usepackage{float}

\renewcommand{\pmod}[1]{\left( \mathrm{ mod\;}#1\right)}

\theoremstyle{plain}
\newtheorem{theorem}{Theorem}

\newtheorem{proposition}{Proposition}[section]

\theoremstyle{remark}
\newtheorem{remark}{Remark}[section]
\newtheorem*{remark*}{Remark}

\theoremstyle{definition}

\newtheorem{example}{Example}[section]

\renewcommand*{\backref}[1]{}
\renewcommand*{\backrefalt}[4]{%
  \ifcase #1 %
No citations.
  \or
(page #2).%
  \else
(pages #2).%
  \fi%
}
\usepackage{cite}

\begin{document}

\title[On non-holonomicity, transcendence and $p$-adic valuations]
{On non-holonomicity, transcendence and $p$-adic valuations}

\author[C. Cobeli, M. Prunescu A. Zaharescu]{Cristian Cobeli, Mihai Prunescu, Alexandru Zaharescu}

\address{CC: ``Simion Stoilow'' Institute of Mathematics of the Romanian Academy,~21 Calea Griviței Street, P.O. Box 1-764, Bucharest 014700, Romania}
\email{cristian.cobeli@imar.ro}

\address{MP: Research Center for Logic, Optimization and Security (LOS), Faculty of Mathematics and Computer Science, University of Bucharest, Academiei 14, Bucharest (RO-010014), Romania; and ``Simion Stoilow'' Institute of Mathematics of the Romanian Academy, Research unit 5, P.O. Box 1-764, Bucharest (RO-014700), Romania.}
\email{mihai.prunescu@imar.ro}

\address{
AZ: Department of Mathematics,
University of Illinois at Urbana-Champaign,
Altgeld Hall, 1409 W. Green Street, Urbana, IL, 61801, USA and 
``Simion Stoilow'' Institute of Mathematics of the Romanian Academy, 
P.O. Box 1-764, RO-014700 Bucharest, Romania}
\email{zaharesc@illinois.edu}

\makeatletter
\@namedef{subjclassname@2020}{%
  \textup{2020} Mathematics Subject Classification}
\makeatother


\begin{abstract}
Let $\nu_q(n)$ be the $p$-adic valuation of $n$. We show that the power series with \mbox{coefficients}  $\nu_q(n)$, respectively $\nu_p(n) \pmod k$, are non-holonomic and not algebraic in characteristic $0$. We find infinitely many rational numbers and infinitely many algebraic \mbox{irrational} numbers for which the values of these series are transcendental. We apply these results to some $p$-automatic sequences, one of them being the period-doubling sequence.  
\end{abstract}
\maketitle

\section{Introduction}
For any prime $p$ we denote by $\nu_p(n)$ the \emph{$p$-adic 
valuation} of the integer $n\ge 1$.
Thus, if $n=p^sm$ and $\gcd(p,m)=1$, then  $\nu_p(n)=s$.
The $p$-adic valuation is essential in expressing any Kalmar elementary function as an arithmetic term. 
As proven by Mazzanti~\cite{Mazz2002} for $p = 2$ (also see Marchenkov~\cite{Mar2006}), 
and recently generalized in  \cite{PSAprep} for all larger primes
(also see~\cite{PS2024}), the value of $\nu_p (n)$ can be written as
\begin{align*}
  \nu_p(n) = \left\lfloor\frac{\big(\gcd(n,p^n)\big)^{n+1}\pmod{(p^{n+1}-1)^2}}{p^{n+1}-1}\right\rfloor.
\end{align*}

According to Mazzanti~\cite[Lemma 3.3 and Lemma 4.2]{Mazz2002}, the functions  $\gcd(a,b)$ and $\binom{2n}{n}$,  can also be expressed through certain such arithmetic formulas.
It is then a consequence of Kummer's theorem that the \emph{Hamming weight} $\wt(n)$, which is defined as the number of digits equal with~$1$ in the binary representation of $n\ge 1$,
is expressible by the arithmetic term
\begin{align*}
  \wt(n) = \nu_2 \binom{2n}{n}.   
\end{align*}
We note that the Hamming weight function is essential in expressing the number of solutions of an exponential Diophantine equation with parameters as an arithmetic term depending on these parameters. 
Following that, it is conceivable that one could optimize the construction of arithmetic terms for Kalmar elementary functions 
if one could find simpler arithmetic terms to express the functions $\nu_2(n)$ and 
$\binom{2n}{n}$. 
Along these lines, our object in this article is to study the $p$-adic valuation $\nu_p(n)$ 
from the point of view of the generating function of the sequence 
$\{\nu_p(n)\}_{n\ge 1}$ for any prime $p$.

For any prime $p$, the \emph{generating function} $V_p(X)$ of
the sequence $\{\nu_p(n)\}_{n\ge 1}$  is the following formal power series in $X$:
\begin{equation}\label{eqDefV}
    V_p(X) = \sum_{n=1}^\infty \nu_p(n)X^n\,.
\end{equation} 

The radius of convergence of $V_p(X)$ at $0$ is $R = 1$.
Then, since $1/2 < R$, and 
$\{\nu_p(n)\}_{n\ge 1}$ 
is a sequence of integers satisfying 
\mbox{$\nu_p(n) \leq \log_p n$} 
for all~\mbox{$n \geq 1$},
according \mbox{to~\cite[Theorem 2]{PSAxiv}}, we have the following uniform representation of the function $\nu_p(n)$ for all primes $p \geq 2$:
\begin{align*}
  \nu_p(n) = \left \lfloor 2^{n^2} V_p(2^{-n}) \right \rfloor \pmod{ 2^{n}}
  \ \text{ for all $n \geq 1$.}  
\end{align*}

However, the results of the present article seriously diminish the hope that one could transform the expression above into some usable tool to compute the $p$-adic valuation. Thus, in Theorem~\ref{Theorem1}  we prove that the function $V_p(X)$ is non-holonomic. Stated differently, not only that the series is not rational or algebraic, but it is far from any combination of familiar analytic functions. Also, we show in Theorem~\ref{Theorem2} that if $X$ is taken from a certain large class of rational numbers then the value $V_p(X)$ is transcendental. 

\smallskip
Our results entail a broader potential and can be applied to the following objects.
Let $q \geq 2$ be an integer that is not necessarily prime.
Using the same notation as for the valuation function, for every integer $n \geq 1$ we let
\begin{align*}
  \nu_q(n) := \max \{ k \in \NN \,:\,  q^k \,|\, n\}.  
\end{align*}
Given this more general function $\nu_q(n)$, 
let $V_q(X)$ be defined as before by formula~\eqref{eqDefV}. 
Then, for any integer $k \ge 2$, let $V_{q,k}(X)$ be the reduced modulo $k$ version of $V_q(X)$, which is formally obtained by reducing all coefficients modulo $k$, that is,
\begin{align}\label{eqDefVpk}
 V_{q,k}(X) := \sum_{n=1}^\infty 
 \big(\nu_q(n) \pmod k\big) X^n\,,    
\end{align}
where the coefficients are understood to be the unique integers in the interval $[0,k-1]$ that are congruent to $\nu_q(n)$ modulo $k$.
We are now ready to state our main results.

\begin{theorem}\label{Theorem1}
The series $V_q(X)$ and $V_{q,k}(X)$ defined by~\eqref{eqDefV} and~\eqref{eqDefVpk}
 are both non-holonomic for every integers $q, k \geq 2$. 
\end{theorem}

\begin{theorem}\label{Theorem2}
Let  $a$ and $b$ be relatively prime positive integers with ratio $a/b<1$.
Then, the real numbers $V_q\big(a/b\big)$ and $V_{q,k}\big(a/b\big)$ are transcendental for all 
integers $q, k \geq 2 $ satisfying condition
\begin{equation}\label{eqTheorem1}
    q > \frac{2\log b}{\log b -\log a}\,.
\end{equation}
Moreover, while $q$ still satisfies condition~\eqref{eqTheorem1}, the complex numbers
$V_q\big(\omega(q,\ell)\cdot a/b\big)$ and $V_{q,k}\big(\omega(q,\ell)\cdot a/b\big)$ are also transcendental
for all roots of unity $\omega(q,\ell)=\exp\big(\frac{2\pi h}{q^\ell}\big)$, 
where $\ell\ge 1$ and $0\le h\le q^\ell-1$.
\end{theorem} 

\begin{theorem}\label{TheoremExtra}
If $\AA$ is a field of characteristic $0$ 
consisting of algebraic numbers only, then the power series
$V_q(X)$ and $V_{q,k}(X)$ viewed as elements of $\AA[[X]]$ 
are not algebraic over~$\AA(X)$. 
\end{theorem}

We point out that Theorems~\ref{Theorem2}
and~\ref{TheoremExtra} are part of certain
more general Theorems~\ref{Theorem4} and~\ref{TheoremExtra1}
that we prove in Sections~\ref{SubsectionLargerClass}
and~\ref{SubsectionProofTheorem3}, respectively.

\section{Non-holonomic series} 

\subsection{\texorpdfstring{The function $V_{q,k}(X)$}{The function V(X)}} 

A formal power series $f = f(X) \in \CC[[X]]$ is called \emph{holonomic} if there are polynomials 
$a_0(X),a_1(X), \dots, a_r(X) $ $ \in \CC [X]$, with the leading polynomial $a_r (X) \neq 0$,  such that the following identity holds
\begin{align*}
   a_r(X) \sdfrac{\dd^r}{\dd X^r} f(X) + a_{r-1}(X) \sdfrac{\dd^{r-1}}{\dd X^{r-1}} f(X) + \dots + a_1 (X) \sdfrac{\dd}{\dd X}f(X) + a_0(X) f(X) = 0. 
\end{align*}
Holonomic series are the generating functions of the holonomic sequences. 
Thus, a sequence $\{c(n)\}_{n\ge 0}$ of complex numbers is called \emph{holonomic} if there are $r\ge 0$ polynomials 
$a_r(X), a_{r-1}(X),\dots,a_0(X) \in \CC [X]$, with the leading polynomial $a_r(X)$ not identically zero,  such that for all $n \geq 0$, 
\begin{align*}
  a_r(n) c(n+r) + a_{r-1}(n) c(n+r-1) + \cdots + a_0(n) c(n) =0.  
\end{align*}
In case that all polynomials $a_r(X),\dots, a_0(X)$
are actually constant complex numbers, which means that they all have degree zero, 
the sequence $\{c(n)\}_{n\ge 0}$ is called $\mathrm{C}$-\emph{recursive} or \emph{constant-recursive}. 
A sequence is \Crec if and only if its generating function is a rational function $A(X) / B(X)$ with $\deg(A) < \deg(B)$ (see Kauers and Paule~\cite{KP2011}, and Stanley~\cite{Sta2012}). 
Examples of representations of \Crec sequences can be found in articles of Prunescu \cite{PruOtherTwo} and with Sauras-Altuzarra~\cite{PSAxiv}.

When viewed as a power series in $\CC[[X]]$, 
$V_q(X)$ is absolutely convergent for~\mbox{$|X|<1$}.
The fact that the radius of convergence is~$1$ can be seen by applying the Cauchy-Hadamard theorem, 
or the root test, by observing that the $q$-adic valuation of $n$ is bounded by $\log n/\log q$ and the equality
holds for all powers of $q$. Therefore
\begin{equation*}
	\limsup_{n\to\infty}\big(\log n\big)^{1/n}
	= \limsup_{n\to\infty} \exp\Big(\sdfrac{\log\log n}{n}\Big)
	=1,
\end{equation*}
and, likewise, $\lim\limits_{n\to\infty}q^{1/n}=1$.

In addition, the series $V_{q,k}(X)$ is absolutely convergent for $|X| < 1$ and the radius of convergence is $1$. Indeed, 
\begin{equation*}
    \limsup_{n \to \infty} 
        \big(\nu_q(n) \pmod k\big)^{1/n}
    =\limsup_{n \to \infty}\,
        (k-1)^{1/n}
    = 1,
\end{equation*}
for all integers $k \geq 2$. 

Next we prove the non-holonomy of both the series $V_q(X)$ and the sequence $\{\nu_q(n)\}_{n\ge 1}$ for all integers $q\ge 2$.
The result is based on a theorem of P\'olya and Carlson
 (see P\'olya~\cite{Pol1916}, Carlson~\cite{Car1921}, or Bieberbach~\cite{Bieber1968})
according to which 
a series with integer coefficients, which is analytic at the origin and converges in the open unit disc, 
is either a rational function or has the open unit circle as  natural boundary. 
For other implications of the theorem of P\'olya and Carlson to non-holonomic functions see Flajolet, Gerhold and Salvy~\cite{Ger2004, FGS2005}.
\begin{proposition}\label{Propositionpnonholonomic}
  If $q \geq 2$, then the series $V_q(X)$ is not holonomic. 
Also, the sequence $\{\nu_q(n)\}_{n\ge 1}$ is not holonomic. 
\end{proposition} 

\begin{proof} 
{\bf Case 1}: $V_q(X)$ is rational. A rational function $A(X) / B(X)$ can be always written as $P(X) + C(X) / B(X)$, where $A(X), B(X), C(X), P(X) \in \CC[X]$ are polynomials, and $\deg(C(X)) < \deg(B(X))$. 
It follows that $\nu_q(n) = x(n) + y(n)$ for all integers $n \ge 0$, 
where $x(n)$ is a sequence with finite support 
and $y(n)$ is a \Crec sequence. 
Therefore there is some integer $N\ge 0$ such that for all 
$n \geq N$, $\nu_q(n) = y(n)$. Let $d \geq 1$ be the order of~$y$ as a C-recursive sequence. It follows that for all $n \geq N$, the sequence $\nu_q(n)$ satisfies the recurrence with constant coefficients
\begin{align*}
a_d \nu_q(n+d) + a_{d-1} \nu_q(n+d-1) + \cdots + a_0 \nu_q(n) =0,  
\end{align*}
 which is satisfied by the sequence $y$. Here $a_d, \dots, a_0 \in \CC$ are some constants.
We observe that for every $k \geq 1$ the minimal distance between two different natural numbers $n_1 \neq  n_2$, with $\nu_q(n_1) \geq k$ and $  \nu_q(n_2) \geq k$, is $q^k$. Take $k$ such that $q^k > d$. Take $m_1 > N$  and $m_2 > N$ such that $\nu_q(m_1) = k$ and $\nu_q(m_2) = k + 1$. We see that $\nu_q(m_1 - i) = \nu_q(m_2 - i)$
for $i = 1, \dots, d$, and this would imply $\nu_q(m_1) = \nu_q(m_2)$, which is a contradiction.

{\bf Case 2}: $V_q(X)$ is not rational.
Since the function $V_q(X)$ has radius of convergence $R = 1$ at $0$, but is not a rational function, it follows from the theorem of P\'olya and Carlson that $V_q(X)$ has infinitely many singularities on the circle $C(0,1)$. A dense set of singularities will be described in the next section. But a function which is analytic at $0$ and holonomic can be continued analytically along any path that avoids the finite set $\cS$ that consists of the roots of the polynomial $a_r(X)$ 
(see Henrici~\cite[{Theorem 9.1}]{Henrici1974}) for the existence of analytic solutions for differential equations. But this does not happen with the function $V_q(X)$, because it cannot be analytically continued outside the disc $D(0,1)$. 
\end{proof} 

\begin{proposition}\label{Propositionpknonholonomic}
If $q, k \geq 2$ are integers, then the series $V_{q,k}(X)$ is not holonomic. Also, the sequence 
$\{\nu_p(n) \bmod k\}_{n\ge 1}$     
is not holonomic. 
\end{proposition} 

\begin{proof}
The proof follows the same steps as in the proof of Proposition~\ref{Propositionpnonholonomic}. Again, by 
the theorem of P\'olya-Carlson, the series is either a rational function or has the unit circle as a natural boundary. 

{\bf Case 1}: If the series was a rational function, then the sequence $\nu_q(n) \pmod k$ would be ultimately a \Crec sequence of some degree $d \geq 1$. There are infinitely many integers $m_1$ and $m_2$ such that 
$\nu_q(m_1 - j) \pmod k = \nu_q(m_2 - j) \pmod k$
for $j = 1, \dots, d$, 
since $\nu_q(m_1 - j) = \nu_q(m_2 - j)$,  
but $\nu_q(m_1) \pmod k \not=  \nu_q(m_2) \pmod k$. 

{\bf Case 2}: 
If the series is not a rational function, the same argument from Proposition~\ref{Propositionpnonholonomic} applies.
\end{proof}

\section{Alternative expressions, identities and singular points}

We write the formal power series $V_q(X)$ with terms separated in such a way that all coefficients are equal to $1$ 
and group them according to the similar types of powers they have, as follows:
\begin{equation*}
    V_q(X) = \big(X^q+X^{2q}+X^{3q}+\cdots\big)
         +\big(X^{q^2}+X^{2q^2}+X^{3q^2}+\cdots\big)
         +\big(X^{q^3}+X^{2q^3}+X^{3q^3}+\cdots\big)+\cdots,
\end{equation*}
so that
\begin{equation}\label{eqVfractions}
  \begin{split}
    V_q(X) &= \sum_{m=1}^\infty X^{mq} + \sum_{m=1}^\infty X^{mq^2} + \sum_{m=1}^\infty X^{mq^3}+\cdots\\
      & = \sdfrac{X^q}{1-X^q} + \sdfrac{X^{q^2}}{1-X^{q^2}} + \sdfrac{X^{q^3}}{1-X^{q^3}} +\cdots.
  \end{split}
\end{equation}

The domain where $V_q(X)$ is analytic cannot be extended outside $|X|<1$
as can 
still be seen from the expression of $V_q(X)$ in~\eqref{eqVfractions},
where the fractions~$\frac{X^{q^k}}{1-X^{q^k}}$ 
tend very fast to zero as~$k$ increases if $|X|<1$, while all of them explode at any root of unity of order a power of $q$.
Geometrically, the poles of the fractions~$\frac{X^{q^k}}{1-X^{q^k}}$ are distributed on the unit circle $\cC\big((0,0),1\big)$ 
as the vertices of regular polygons, and their union becomes dense on $\cC\big((0,0),1\big)$ as $k$ tends to infinity. 

\smallskip
Now let $h,\ell$ be fixed positive integers and let $\omega$ be the root of unity 
\begin{equation*}
   \omega:=\exp\Big(\sdfrac{2\pi i h}{q^\ell}\Big).
\end{equation*}
Then
\begin{equation}\label{eqVfromega}
  \begin{split}
    V_q(\omega X) = \sdfrac{\omega^qX^q}{1-\omega^qX^q} 
      + \sdfrac{\omega^{q^2}X^{q^2}}{1-\omega^{q^2}X^{q^2}}  +\cdots
      + \sdfrac{\omega^{q^\ell}X^{q^\ell}}{1-\omega^{q^\ell}X^{q^\ell}} 
      + \sdfrac{\omega^{q^{\ell+1}}X^{q^{\ell+1}}}{1-\omega^{q^{\ell+1}}X^{q^{\ell+1}}} 
      +\cdots.
  \end{split}
\end{equation}
Here, in the last two fractions shown and in all the others that follow, 
the factors with $\omega$ at different powers become equal to $1$, 
so that the fractions in the expansion are exactly identical to those of $V_q(X)$ in~\eqref{eqVfractions}. 
As a consequence, the difference $V_q(X)-V_q(\omega X)$ is a rational function.
We summarize the result in the following proposition.
\begin{proposition}\label{Proposition1}
Let $q \geq 2$ and  let $h$ and $\ell$ be positive integers.
Suppose $h$ is relatively prime to $q$ and let
$\omega=\exp\Big(\frac{2\pi i h}{q^\ell}\Big)$.
Then
\begin{equation}\label{eqVDifferenceRationsl}
  \begin{split}
    V_q(\omega X) - V_q(X) &= \sdfrac{\omega^qX^q}{1-\omega^qX^q} 
      + \sdfrac{\omega^{q^2}X^{q^2}}{1-\omega^{q^2}X^{q^2}}  +\cdots
      + \sdfrac{\omega^{q^{\ell-1}}X^{q^{\ell-1}}}{1-\omega^{q^{\ell-1}}X^{q^{\ell-1}}}\\
      &\quad - \sdfrac{X^q}{1-X^q} - \sdfrac{X^{q^2}}{1-X^{q^2}} -\cdots 
        -\sdfrac{X^{q^{\ell-1}}}{1-X^{q^{\ell-1}}}\,.
  \end{split}
\end{equation}
\end{proposition}
\begin{remark}\label{Remark1}
    Two examples:
    \begin{enumerate}
     \item 
        If $\ell=1$, then $\omega = \exp\big(\frac{2\pi i}{q}\big)$, so $\omega^q=1$, 
        which implies $V_q(\omega X) = V_q(X)$,
    equality that is included by convention in~\eqref{eqVDifferenceRationsl} as well.
      \item 
       If $p=2$, we get a nontrivial result with roots of $-1$ instead of roots of unity. Indeed, 
       with $\omega = i$, we have
   \begin{equation*} 
     V_2(iX)- V_2(X) = \sdfrac{-X^2}{1+X^2} - \sdfrac{X^2}{1-X^2}
                    = \sdfrac{-2X^2}{1-X^4}\,.  
\end{equation*}    
  Taking $\omega$ as a root of unity of order greater than $4$, 
a similar outcome is also achieved with the difference $V_q(\omega X) - V_q(X)$ 
being a rational function over $\QQ(\omega)$.
    \end{enumerate}
\end{remark} 

In order to treat the series $V_{q,k}(X)$, 
we need to introduce 
the \textit{bitty series} $f_r(X)=f_{q,r}(X)$
defined by
\begin{align}\label{eqDeffsX}
  f_r(X) := \sum _{m=1}^\infty X^{mq^r} 
  = \frac{X^{q^r}}{1 - X^{q^r}}
  \ \text{ for any $r \geq 1$. }
\end{align}

\begin{proposition}\label{Proposition2}
For every integers $q, k \geq 2$, we let 
$V_{q,k}(X)$ be defined by~\eqref{eqDefVpk} and for every integers
$r\ge 1$ we let $f_r(X)$ be defined by~\eqref{eqDeffsX}.
Then we have:
\begin{equation}\label{eqVpk}
    \begin{split}
     V_{q,k}(X) &= f_1(X) + \dots + f_{k-1}(X) + (1-k)f_k(X) +\\
     &\quad + f_{k+1}(X) + \dots + f_{2k-1}(X) + (1-k) f_{2k} (X) +  \\
     &\quad + f_{2k+1}(X) + \dots + f_{3k-1}(X) + (1-k) f_{3k} (X) + \dots \\
     &= \sum_{j = 1}^\infty a_j f_j(X),
     \end{split}
\end{equation}
    where
\begin{align}\label{eqDefaj}
  a_j = \begin{cases}
    1, & \textrm{ if }\ k \nmid j; \\[2pt] 
    1-k, & \textrm{ if  }\ k \mid j.
    \end{cases}
\end{align}
Also, if $h$ and $\ell$ are positive integers, $h$ is relatively prime to $q$ and $\omega=\exp\Big(\frac{2\pi i h}{q^\ell}\Big)$, then $V_{q,k}(\omega X) - V_{q,k}(X)$ is a rational function. 
\end{proposition} 

\begin{proof}
For every complex number $x$, with $|x|<1$, since the series 
$\sum_{n\ge 1} \big(\nu_q(n) \pmod k\big) x^n$ is absolutely convergent, we can rearrange and regroup  its terms as follows:
\begin{align*}
     V_{q,k}(x) &= f_1(x) + \dots + f_{k-1}(x) + (1-k)f_k(x) +\\
     &\quad + f_{k+1}(x) + \dots + f_{2k-1}(x) + (1-k) f_{2k} (x) +  \\
     &\quad + f_{2k+1}(x) + \dots + f_{3k-1}(x) + (1-k) f_{3k} (x) + \cdots \\
     &= \sum_{j = 1}^\infty a_j f_j(x),
\end{align*}
where the coefficients $a_j$ are defined as in~\eqref{eqDefaj}.

Let us examine the expansion in the middle of 
the right-hand side of formula~\eqref{eqVpk} and 
compute the coefficient of $X^{mq^r}$, 
where $q \nmid m$. 
Let $r = dk + v$, where $0 \leq v = r \pmod k < k$. We observe that $X^{mq^r}$ arises in all $f_j(X)$ with $0 < j \leq r$. The sum of its coefficients equals:
\begin{align*}
\overbrace{
  \big( \underbrace{1 + \cdots + 1}_{k-1 \text{ $1$'s}}  + (1-k)\big) + 
  \dots + \big( \underbrace{1 + \cdots + 1}_{k-1 \text{ $1$'s}} 
  + (1-k)\big)
  }^{d \text{ terms}}
  +  \underbrace{1 + \cdots + 1}_{v \text{ $1$'s}} = v.
\end{align*}
Thus, the coefficient of  $X^{mq^r}$ on the right-hand side
of formula~\eqref{eqVpk} is 
\begin{align*}
  v = r \pmod k = \nu_q(m q^r) \pmod k.  
\end{align*}

The fact that $V_{q,k}(\omega X) - V_{q,k}(X)$ 
is a rational function follows from the above 
in a similar way as the rationality of $V_{q}(\omega X) - V_{q}(X)$ followed in the proof of Proposition~\ref{Proposition1}.
\end{proof} 

Let $k\ge 2$ be a fixed integer. 
For every integer $j\ge 1$, we let $A_{k,j} (X)$ denote the sum
\begin{align*}
    A_{k,j} (X) := f_{kj+1}(X) + \cdots+ f_{kj+k-1}(X) 
             + ( 1-k ) f_{k(j+1)}(X),
\end{align*}
where $f_{r}(X)$ is defined by~\eqref{eqDeffsX} for $r\ge 1$.

We observe that  $f_n(x) > f_{n+t}(x)$ for all real numbers $x\in (0,1)$ and all integers~$t\ge 1$. 

We also observe that for any integer $j\ge 0$ and for any $x\in (0,1)$ we have:
\begin{enumerate}
    \item [(i)]
 $A_{k,j}(x) > 0$; and
    \item  [(ii)]
  $A_{k,j}(x) > A_{k, j+1}(x) $.
\end{enumerate}

We say that a sequence 
\begin{align*}
  W(X) := \sum_{i \geq 1} a_i f_i(x)
\end{align*}
\textit{has positive segments} if there exists a strictly increasing sequence $\{n_j\}_{j \geq 1}$, such that
\begin{align*}
    B_j(x) &:= a_{n_j} f_{n_j}(x) + a_{n_j + 1} f_{n_j + 1}(x) + \dots + a_{n_{j+1} -1} f_{n_{j+1} -1}(x) > 0,  \\
    \intertext{and}
    B_{j}(x) &< B_{j+1}(x)
\end{align*}
for all $j \geq 1$ and for all $x \in (0,1)$.

\begin{proposition}\label{PropPositiveQueues}
    Let $q \geq 2$ and $k \geq 2$ be integers. Then the series $V_q(X)$ and $V_{q,k}(X)$ have positive segments.
\end{proposition}

\begin{proof}
For $V_q(X)$ the sequence of indices is $n_j = j$, and $B_j(x) = f_j(x)$ for all $j \geq 1$. For $V_{q,k} (X)$ the sequence of indices is $n_j = kj+1$. In this case, $B_j(x) = A_{k,j}(x)$  for all $j \geq 1$.
\end{proof}

\section{Proofs of Theorems~\ref{Theorem2} and~\ref{TheoremExtra}}\label{SectionPtoofTheorem23}
Before proving Theorem~\ref{Theorem2}, we first outline it
with a particular example.

\subsection{A not so random transcendental encounter} 
Let $q=3$ and let $b$ be some ordinary integer, for instance let $b=2024$. Our aim is to show that~$\alpha=V_q(1/b)$ is transcendental.
For this, we will make use of a result by Roth from 1955, which we write in the following form.
\begin{theorem}[{Roth~\cite{Roth1955}}]\label{TheoremRoth}
   Let $\alpha$ be a real number and let $\delta>0$.
Suppose there exists a sequence of distinct rational numbers 
$\{A_n/B_n\}_{n\ge 1}$ and an integer $n_0$
such that
\begin{equation*}
  \begin{split}
     \Big|\alpha -\sdfrac{A_n}{B_n}\Big| < \sdfrac{1}{B_n^{2+\delta}}\ \ \text{ for $n\ge n_0$.}
  \end{split}
\end{equation*}
Then $\alpha$ is transcendental.
\end{theorem}

Since the series $\alpha=V_3\big(1/b\big)$ is absolutely convergent, we can regroup conveniently and add the terms so that it can be expressed as a sum with the generic terms:
\begin{align*}
     \sdfrac{{b}^{-3^k}}{
     1-{b}^{-3^k}} 
     = \sdfrac{1}{b^{3^k}-1}\ \text{ for $k\ge 1$.}
\end{align*}
Then, as $n$ tends to infinity, we have:
\begin{align}\label{eqalpha2024}
     \alpha &= \sdfrac{1}{{b^{3}-1}}
     + \sdfrac{1}{b^{3^{2}}-1}+\cdots
     + \sdfrac{1}{b^{3^{n}}-1}
     + O\bigg(\sdfrac{1}{b^{3^{n+1}}-1}\bigg)\,,
\end{align}
so that
\begin{equation*}
  \begin{split}
     \alpha &= \sdfrac{A_n}{B_n} + 
     O\bigg(\sdfrac{1}{b^{3^{n+1}}-1}\bigg)\,,
  \end{split}
\end{equation*}
where $B_n=b^{3^n}-1$ and $A_n$ 
is the corresponding numerator of the ratio that adds up all fractions in 
the main term of the approximation of $\alpha$ in~\eqref{eqalpha2024}.
Taking into account the size of the 
constant involved in the big $O(\cdot)$ 
estimate, we find that the following
explicit inequality holds
\begin{equation*}
  \begin{split}
     \Big|\alpha - \sdfrac{A_n}{B_n}\Big|
     \le \sdfrac{1}{B_n^{5/2}}\ \ 
     \text{ if $n\ge n_0$}\,,
  \end{split}
\end{equation*}
for some positive integer $n_0$.
This means that the hypotheses of Theorem~\ref{TheoremRoth} are 
satisfied with~$\delta = 1/2$, and consequently it follows that 
$\alpha$ is transcendental.

\subsection{A larger class of transcendental numbers}\label{SubsectionLargerClass}
The arguments from the previous section can be adapted to obtain a large class of transcendental numbers, which we will do in the following.

\begin{theorem}\label{Theorem4}
Let $q, M\ge 2$ be integers and let 
$a$ and $b$ be relatively prime positive integers with ratio $a/b<1$.
We let $W(X)$ be defined by
\begin{align*}
  W(X) := \sum _{j = 1}^\infty a_j f_j(X) 
  = \sum _{j = 1}^\infty a_j \sdfrac{X^{q^j}}{1-X^{q^j}},
\end{align*}
where all $a_j \in \ZZ$ satisfy $|a_j| < M$. 
If $W(X)$ is not a rational function and has positive segments, then $W\big(a/b\big)$ is transcendental for all $q \geq 2$ for which
\begin{align*}
  q > \frac{2\log b}{\log b -\log a}\,.
\end{align*}
\end{theorem} 
\begin{proof}

Let $q \geq 2$ and let $a<b$ be relatively prime positive integers. 
Then
\begin{equation}\label{eqVabE}
    W(a/b) = a_1 \sdfrac{\big(a/b\big)^{q}}{1-\big(a/b\big)^{q}}
    + a_2 \sdfrac{\big(a/b\big)^{q^2}}{1-\big(a/b\big)^{q^2}}
    +\cdots 
    + a_n \sdfrac{\big(a/b\big)^{q^n}}{1-\big(a/b\big)^{q^n}}
     +E_{p,n}(a,b),
\end{equation}
where $E_{q,n}(a,b)$ collects the sum of all fractions of higher order from the corresponding analogue of \eqref{eqVfractions}.
In order for $E_{q,n}(a,b)$ to be small, we take $n$ sufficiently large such that
\begin{equation}\label{eqboundn12}
    \big(a/b\big)^{q^n}<1/2,
\end{equation}
and the corresponding queue is positive. (All the indices considered in the proof will belong to the sub-sequence of positive segments.) The inequality $(a/b)^{q^n} < 1/2$ implies that for $k>n$, the general term of the sum $E_{q,n}(a,b)$ is
\begin{equation*}
  \left | a_k \sdfrac{\big(a/b\big)^{q^k}}{1-\big(a/b\big)^{q^k}} \right | \le 2 M \big(a/b\big)^{q^k}.
\end{equation*}
Using this inequality in the definition of $E_{q,n}(a,b)$,
we can further bound it with the sum of a geometric series as follows
\begin{equation*}
  \begin{split}
    \big|E_{q,n}(a,b)\big| &\le 
    2M\big(a/b\big)^{q^{n+1}} + 2M\big(a/b\big)^{q^{n+2}} 
    + 2M\big(a/b\big)^{q^{n+3}}+\cdots\\
     & = 2M\big(a/b\big)^{q^{n+1}}\bigg(
    1 + \big(a/b\big)^{q^{n+2}-q^{n+1}}
    + \big(a/b\big)^{q^{n+3}-q^{n+1}} +\cdots\bigg)
    \,.
  \end{split}
\end{equation*}
Here, the general term is bounded using 
\begin{equation*}
    x^{q^{n+k}-q^{n+1}}
    \le \big(x^{q^n}\big)^{k-1} 
    \le 2^{-k+1},
\end{equation*}
for $k\ge 1$, provided that $0< x^{q^n}<1/2$.
Then, with $x=a/b$, as the assumption~\eqref{eqboundn12} holds, 
it follows 
\begin{equation}\label{eqBoundE}
  \begin{split}
 \big|E_{q,n}(a,b)\big| 
 \le 2M\big(a/b\big)^{q^{n+1}}
     \big(1+2^{-1}+2^{-2}+\cdots\big)
 \le 4M \big(a/b\big)^{q^{n+1}}.
  \end{split}
\end{equation}

Note that all denominators of the fractions in the main term of 
$W(a/b)$ in \eqref{eqVabE}, viewed as polynomials in $a/b$, divide
$1-\big(a/b\big)^{q^n}$.
Next, we define $A_n$ and $B_n$ as the numerator and the denominator
of the sum of fractions of the principal term, that is,
\begin{equation}\label{eqAB}
 A_n := B_n\sum_{k=1}^n a_k \sdfrac{a^{q^k}}{b^{q^k}-a^{q^k}},\ \ 
 \text{ and }\ \ B_n := b^{q^n}-a^{q^n}.
\end{equation}

On combining~\eqref{eqVabE}, \eqref{eqAB} and~\eqref{eqBoundE}, we obtain
\begin{equation}\label{eqVAB}
    \Big|W(a/b)-\sdfrac{A_n}{B_n}\Big|\le 4 M \big(a/b\big)^{q^{n+1}}.
\end{equation}

Since we want to make use of Roth's theorem~\ref{TheoremRoth}, 
we must ensure that the hypotheses are met, 
of which, taking the inequality \eqref{eqVAB} into account, it remains to see that 
there exists $\delta>0$ such that
\begin{equation}\label{eqBtare}
      4 M \big(a/b\big)^{q^{n+1}} < \sdfrac{1}{B_n^{2+\delta}}.
\end{equation}
Replacing $B_n$ by its definition in~\eqref{eqAB},
we find that for~\eqref{eqBtare} to hold,
it is enough to show that there exists $\delta>0$ and an integer $n_0$
such that 
\begin{equation}\label{eqBtare2}
      4 M \big(b^{q^{n}}\big)^{2+\delta} a^{q^{n+1}} < b^{q^{n+1}}\ \
      \text{ for $n\ge n_0$}.
\end{equation}
Taking logarithms and dividing both sides by $q^n$,~\eqref{eqBtare2} is equivalent with
\begin{equation}\label{eqBtare3}
    \sdfrac{\log (4M)}{q^{n}}
      +(2+\delta)\log b +q\log a < q\log b\ \
      \text{ for $n\ge n_0$}.
\end{equation}
As the inequality in~\eqref{eqBtare3} is strict,
it is fully achieved when $n$ is sufficiently large
if it holds without the single term that depends on $n$, that is, if 
there exists $\delta>0$ such that
\begin{equation}\label{eqBtare4}
      (2+\delta)\log b +q\log a < q\log b.
\end{equation}
Turning back to the exponential form, we see that~\eqref{eqBtare4}
is equivalent with
\begin{equation}\label{eqBtare5}
      b^{2+\delta}a^q< b^q.
\end{equation}
Once again, due to the strict inequality~\eqref{eqBtare5}, 
the existence of a sufficiently small 
$\delta>0$ is ensured if the inequality holds even without the factor $b^\delta$ 
on the left side, that is if 
\begin{equation}\label{eqBtare6}
      b^{2}a^q < b^q.
\end{equation}
But this is exactly the initial requirement~\eqref{eqTheorem1} first imposed on $a$ and $b$ in the hypothesis of Theorem~\ref{Theorem2}. 

As we have seen following 
now backwards transformations~\eqref{eqBtare5}--\eqref{eqBtare},
this ensures that the last requirement~\eqref{eqBtare} 
needed for the application of Roth's theorem~\ref{TheoremRoth}
is now fulfilled. As a consequence, we have proved that if integers $a,b$
satisfy the hypotheses of Theorem~\ref{Theorem4}, then 
$W\big(a/b\big)$ is transcendental 
except for at most a finite number of integers~$q$.
To be more specific, equation~\eqref{eqBtare6} indicates that 
the potential exceptions are bounded in size by
$2\log b/(\log b-\log a)$.
This concludes the proof of Theorem~\ref{Theorem4}.
\end{proof}

\medskip
\subsection{Proof of Theorem~\ref{TheoremExtra}}\label{SubsectionPProofTheorem2}
If we take $W(X) = V_q(X)$, respectively $W(X) = V_{q,k}(X)$, we obtain the results stated in Theorem \ref{Theorem2} for the transcendence of their values computed in $a/b$. 
The last part of the proof of Theorem~\ref{Theorem2}, 
that regarding the transcendence of $V_q\big(\omega a/b\big)$, respectively $V_{q, k}\big(\omega a/b\big)$, where
$\omega$ is a certain root of unity,
follows from Proposition~\ref{Proposition1} and from Proposition \ref{Proposition2}.
Indeed, by contradiction, if $V_q\big(\omega a/b\big)$ were algebraic, 
then since we already know that $V_q\big(a/b\big)$ is transcendental, 
it would follow that the difference 
$V_q\big(\omega a/b\big)- V_q\big(a/b\big)$
is also transcendental.
Yet, this directly contradicts the conclusion of Proposition~\ref{Proposition1}, 
which implies that the difference 
$V_q\big(\omega a/b\big)- V_q\big(a/b\big)$ is an algebraic number. The same for $V_{q,k}\big (\omega a/b \big )$.
This concludes the proof of Theorem~\ref{Theorem2}.
\qed

\subsection{A generalized version of Theorem~\ref{TheoremExtra}}\label{SubsectionProofTheorem3}
The following is a more general result that includes the
statement that  $V_q(X)$ and $V_{q,k}(X)$ are not algebraic over $\AA(X)$.
\begin{theorem}\label{TheoremExtra1}
Let $q \geq 2$ and let $M \ge 2$ be a constant integer. 
Let
\begin{align*}
  W(X) = \sum_{j = 1}^\infty a_j f_j(X) 
  = \sum _{j=1}^\infty a_j \sdfrac{X^{q^j}}{1-X^{q^j}}, 
\end{align*}
where all $a_j$ are integers satisfying $|a_j| < M$ 
such that $W(X)$ is not a rational function 
and has positive segments.  
Suppose that $\AA$ is a field of characteristic $0$ 
consisting of algebraic numbers only. 
Then $W(X) \in \AA[[X]]$ is not algebraic over 
the field of rational functions~$\AA(X)$. 
\end{theorem}

\begin{proof}
Suppose that $W(X)$ satisfies an equation
\begin{align*}
 g_r(X) W^r(X) + g_{r-1}(X)W^{r-1}(X) + \dots + g_0(X) = 0,   
\end{align*}
where $g_r(X) \neq 0$ and $g_j(X) \in \AA(X)$ 
for $0\le j\le r$.
The rational functions $g_j(X)$ are not defined in a finite set of algebraic numbers, which are poles for at least one of them. Then there will always exist a rational number $a/b \in(0,1)$ that satisfies the hypotheses of Theorem~\ref{Theorem4} for $W(X)$, and which is different from all the finitely many exceptional algebraic numbers in which at least one rational function $g_j(X)$ is not defined. 
Consequently
\begin{align*}
  g_r(a/b) W(a/b)^r + g_{r-1}(a/b)W(a/b)^{r-1} + \cdots + g_0(a/b) = 0.  
\end{align*}
But this would mean that $W(a/b)$ is algebraic, which contradicts Theorem~\ref{Theorem4}. 
This contradiction invalidates the assumption made, thus concluding the proof of Theorem~\ref{TheoremExtra1}.
\end{proof}

\section{Application to some automatic sequences} 

Let $\cS, \cT$ be finite sets and let $w \geq 2$ be integer. 
A \emph{finite deterministic automaton with output} is a structure $\Auto = (\cS, w, \cT, s_0, \delta, \pi)$ defined as follows: the \emph{state set} is $\cS$ (also called the \emph{input alphabet}) is the set $\{0, 1, \dots, w-1\}$, the \emph{output alphabet} is $\cT$, the \emph{start state} 
is~$s_0$,  the \emph{transition function} is $\delta : \cS \times \{0, 1, \dots, w-1\} \rightarrow \cS $, and the \emph{output  function} is~$\pi : \cS \rightarrow \cT$. 

Let $\{0, 1, \dots, k-1\}^*$ be the set of all finite words built up with $0, 1, \dots, k-1$. 

We let $\tilde \delta : \cS \times \{0, 1, \dots, k-1\}^* \rightarrow \cS$ be the \emph{transitive closure} 
of the transition function~$\delta$ be defined as follows.
If the word $c \in \{0, 1, \dots, k-1\}^* $ has length $|c| = 1$ then $\tilde \delta(s_0, c) = \delta(s_0, c) $. If  $c \in \{0, 1, \dots, k-1\}^* $ is any word and $\tilde \delta(s_0, c) \in \cS$ is known, then 
\begin{align*}
  \tilde \delta (s_0, ac) 
  = \delta\big(\tilde \delta(s_0, c), a\big)\  
  \text{ for every letter 
      $a \in \{0, 1, \dots, k-1\}$.}
\end{align*}

The sequence $a : \NN \rightarrow \cT$ is called $w$-\emph{automatic} if there is a finite deterministic automaton $\Auto$ such that $a(n) = \pi\big(\tilde \delta (s_0, n_w)\big)$
for all $n \in \NN$, where the word $n_w \in \{0, 1, \dots, w-1\}^*$ denotes the representation of $n$ in base $w$. 

\begin{theorem}\label{PropAutomaticNuModK}
Let $w, k \geq 2$ be integers. 
Consider the sequence 
\mbox{$a: \NN \rightarrow \{0, 1, \dots, k-1\}$} given by
\begin{align*}
  a(n) = \nu_w(n) \pmod k.   
\end{align*}
(Here, the representative of the residue class modulo $k$ is taken from the interval $[0,k-1]$.)
Then the sequence $\{a(n)\}_{n\ge 1}$
is $w$-automatic.
\end{theorem}

\begin{proof}
We consider the automaton 
$\Auto = (\cS, w, \{0, 1, \dots, k-1\}, s_0,\delta,\pi)$ 
defined as follows. 

Let the state set $\cS$ be given by 
\begin{align*}
  \cS = \big\{ \balpha_k(j,b)  : \balpha_k(j,b) = j \pmod k 
  \text{ for }
  j \in \{0, \dots, k-1\}  \text{ and } b \in \{0, 1\} \big\}.  
\end{align*}

The start state is $s_0 = \balpha_k(0,0)$.
Next, the transition function $\delta$ works based on the set of rules:
\begin{align*}
  \delta\big( \balpha_k(j,0), 0\big) 
    & = \balpha_k(j+1,0), \\
    \delta\big( \balpha_k(j,0),x\big) 
    & = \balpha_k(j,1), \\
    \delta\big( \balpha_k(j,1),a\big) 
    & = \balpha_k(j,1),
\end{align*} 
for all $x = 1, 2, \dots, k-1$ and for all $a = 0, 1, \dots, k-1$,
and
the output function $\pi$ is defined~by 
\begin{align*}
  \pi\big(\balpha_k(j,b)\big) = j.
\end{align*} 

Then, let us observe that 
\begin{align*}
  \pi\big(\tilde \delta(s_0, n_w)\big) 
  = \nu_w(n) \pmod k\ \text{ for every $n \geq 1$.}
\end{align*} 
This concludes the proof of Theorem~\ref{PropAutomaticNuModK}.
\end{proof}

A result of Christol et al.~\cite[Theorem 1]{CKM1980} (see also Allouche and \mbox{Shallit\cite[Theorem 12.2.5]{AutSeq}})
provides an algebraic characterisation of the 
$q$-automatic sequences, where $q = p^s$ is a prime-power. 
Let $\FF_q$ be a finite field that is large enough
such that there exists a one-to-one function 
$\beta : \cT \rightarrow \FF_q$. 
Consider the following formal power series associated with the sequence  $\{a(n)\}_{n\ge 0}$:
\begin{align*}
  a(X) = \sum _{j = 0}^\infty 
         \beta\big(a(j)\big)X^j \in \FF_q[[X]].
\end{align*}
Then the sequence $\{a(n)\}_{n\ge 0}$ is $q$-automatic if and only if $a(X)$ is algebraic over the field of rational functions $\FF_q(X)$. It is also known that if $p$ is a prime number, a sequence is $p$-automatic if and only if it is $p^s$-automatic for each prime-power $p^s$. (For the last  result, see Allouche and Shallit~\cite[Theorem 6.6.4]{AutSeq}.)
Also, according to 
Denef and \mbox{Lipshitz~\cite[Theorem 4.1]{DenefLipshitz}}, 
if $p^s$ is a prime-power, a sequence $a : \NN \rightarrow \ZZ / p^s \ZZ$ 
is $p$-automatic if and only if there is a sequence $b : \NN \rightarrow \ZZ_p$, where $\ZZ_p$ is the ring of $p$-adic integers, such that 
for all $n \in \NN$, $a(n) \equiv b(n) \pmod {p^s}$ and the 
power series $\sum b(j)X^j$ is algebraic over~$\ZZ_p[[X]]$.

Let now $p$ be a prime number and 
let $s \geq 1$ and $k \geq2 $ be integers. By $\nu_{p^s}(n)$ we understand the largest integer
$v$ such that $p^{sv}$ divides $n$. 
(If $s = 1$, this is $\nu_{p^s} = \nu_p$, the classical $p$-adic valuation.) 
We consider the formal power series
\begin{align*}
  V_{p^s,k} (X) = \sum _{j \geq 1} 
  \big(\nu_{p^s}(j) \pmod k\big) X^j  
\end{align*}
over various rings, which remain to be specified. 
This object has a series of properties that we list 
next.
\begin{theorem}\label{propAutomatic}
Let $p$ be a prime number 
and let $s \geq 1$, and $k \geq 2$ be integers. 
The following statements hold true:
\begin{enumerate}
\renewcommand{\theenumi}{\roman{enumi}}
\item\label{part81} 
The sequence $\{a(n)\}_{n\ge 1}$ given by $a(n) = \nu_{p^s}(n) \pmod k$ is both $p^s$-automatic and \mbox{$p$-automatic}. 
\item\label{part82} 
For every $q = p^s \geq k$ and for every one-to-one function $\beta : \{0, \dots, k\} \rightarrow \FF_q$, the series 
\begin{align*}
  \sum _{j = 0}^\infty 
  \beta\big(\nu_{p^s}(j) \pmod k\big)X^j 
  \in \FF_{q}[[X]]  
\end{align*}
is algebraic over $\FF_{q}(X)$. 
\item\label{part83} 
There is a sequence of $p$-adic integers 
$b : \NN \rightarrow \ZZ_p $ such that the formal power series $\sum b(j) X^j \in \ZZ_p [[X]]$ is algebraic over 
$\ZZ _p (X)$ 
(where $\ZZ_p$ is the ring of $p$-adic integers) and such that
\begin{align*}
 \big(\nu_{p^s}(j) \pmod k\big) \equiv b(j) \pmod {p^s}\ \text{for all $j \geq 1$.}   
\end{align*}
\item\label{part84}  For any field $\AA$ of characteristic $0$ consisting of algebraic numbers only, the power series $V_{p^s, k}(X) \in \AA[[X]]$ is not algebraic over $\AA(X)$.
\end{enumerate}
\end{theorem} 

We observe that the last item is also related to Theorem~\ref{Proposition2}. 

\begin{proof}
To prove part~\eqref{part81}, we put $w = p^s$ in Theorem~\ref{PropAutomaticNuModK} and recall that for every prime~$p$,  a sequence is 
$p^s$-automatic for some $s \geq 2$ if and only if it is 
$p$-automatic. 

Part~\eqref{part82} follows from the theorem of Christol.

Part~\eqref{part83} follows from the theorem of Denef and Lipshitz. 

Part~\eqref{part84}  follows from Theorem \ref{TheoremExtra}.
\end{proof}    

\begin{example}
The most typical example of a sequence 
like those covered by Theorem~\ref{propAutomatic} is that of
the \emph{period-doubling sequence}, 
$0, 1, 0, 0, 0, 1, 0, 1$, 
$0, 1, 0, 0, 0, 1, 0, 0, 0, 1, 0, 0,\dots$,
which is defined by
\begin{align*}
   a(n) = \nu_2(n) \pmod 2, \ \text{ for $n\ge 1$.}
\end{align*}
We mention that among other properties it has, 
$\{a(n)\}_{n\ge 1}$
provides the Von Koch snowflake walk in the plane~\cite[{\href{https://oeis.org/A096268}{A096268}}]{oeis}. 
\end{example}

\vspace{11mm}


\begin{thebibliography}{10}

\bibitem{AutSeq}
Jean-Paul Allouche and Jeffrey Shallit.
\newblock {\em Automatic sequences. {Theory}, applications, generalizations}.
\newblock Cambridge: Cambridge University Press, 2003.
\newblock \href {https://doi.org/10.1017/CBO9780511546563} {\path{doi:10.1017/CBO9780511546563}}.

\bibitem{Bieber1968}
Ludwig Bieberbach.
\newblock {\em Funktionentheorie}.
\newblock Johnson, New York, 1968.

\bibitem{Car1921}
Fritz Carlson.
\newblock {\"U}ber {Potenzreihen} mit ganzzahligen {Koeffizienten}.
\newblock {\em Math. Z.}, 9:1--13, 1921.
\newblock \href {https://doi.org/10.1007/BF01378331} {\path{doi:10.1007/BF01378331}}.

\bibitem{CKM1980}
G.~Christol, T.~Kamae, Michel Mend{\`e}s~France, and G{\'e}rard Rauzy.
\newblock Suites alg{\'e}briques, automates et substitutions.
\newblock {\em Bull. Soc. Math. Fr.}, 108:401--419, 1980.
\newblock \href {https://doi.org/10.24033/bsmf.1926} {\path{doi:10.24033/bsmf.1926}}.

\bibitem{DenefLipshitz}
Jan Denef and Leonard Lipshitz.
\newblock Algebraic power series and diagonals.
\newblock {\em J. Number Theory}, 26:46--67, 1987.
\newblock \href {https://doi.org/10.1016/0022-314X(87)90095-3} {\path{doi:10.1016/0022-314X(87)90095-3}}.

\bibitem{FGS2005}
Philippe Flajolet, Stefan Gerhold, and Bruno Salvy.
\newblock On the non-holonomic character of logarithms, powers, and the $n$th prime function.
\newblock {\em The electronic journal of combinatorics}, 11(2), (The Stanley Festschrift volume):1--16, 2005.
\newblock \href {https://doi.org/10.37236/1894} {\path{doi:10.37236/1894}}.

\bibitem{Ger2004}
Stefan Gerhold.
\newblock On some non-holonomic sequences.
\newblock {\em Electron. J. Combin.}, 11(1):Research Paper 87, 8, 2004.
\newblock \href {https://doi.org/10.37236/1840} {\path{doi:10.37236/1840}}.

\bibitem{Henrici1974}
Peter Henrici.
\newblock {\em Applied and computational complex analysis, vol. 2}.
\newblock John Wiley, New York, 1974.

\bibitem{KP2011}
Manuel Kauers and Peter Paule.
\newblock {\em The concrete tetrahedron. {Symbolic} sums, recurrence equations, generating functions, asymptotic estimates}.
\newblock Texts Monogr. Symb. Comput. New York, NY: Springer, 2011.
\newblock \href {https://doi.org/10.1007/978-3-7091-0445-3} {\path{doi:10.1007/978-3-7091-0445-3}}.

\bibitem{Mar2006}
S.~S. Marchenkov.
\newblock Superpositions of elementary arithmetic functions.
\newblock {\em Diskretn. Anal. Issled. Oper., Ser. 1}, 13(4):33--48, 2006.

\bibitem{Mazz2002}
Stefano Mazzanti.
\newblock Plain bases for classes of primitive recursive functions.
\newblock {\em MLQ Math. Log. Q.}, 48(1):93--104, 2002.
\newblock \href {https://doi.org/10.1002/1521-3870(200201)48:1<93::AID-MALQ93>3.0.CO;2-8} {\path{doi:10.1002/1521-3870(200201)48:1<93::AID-MALQ93>3.0.CO;2-8}}.

\bibitem{oeis}
{OEIS Foundation Inc\!\!}
\newblock The {O}n-{L}ine {E}ncyclopedia of {I}nteger {S}equences, 2023.
\newblock Published electronically at \url{http://oeis.org}.

\bibitem{Pol1916}
Georg P\'olya.
\newblock {\"U}ber {P}otenzreihen mit ganzzahligen {K}oeffizienten.
\newblock {\em Math. Ann.}, 77(4):497--513, 1916.
\newblock \href {https://doi.org/10.1007/BF01456965} {\path{doi:10.1007/BF01456965}}.

\bibitem{PruOtherTwo}
Mihai Prunescu.
\newblock {On other two representations of the C-recursive integer sequences by terms in modular arithmetic}, 2024.
\newblock \href {https://arxiv.org/abs/2406.06436} {\path{arXiv:2406.06436}}.

\bibitem{PSAxiv}
Mihai Prunescu and Lorenzo Sauras-Altuzarra.
\newblock On the representation of {C}-recursive integer sequences by arithmetic terms, 2024.
\newblock \href {https://arxiv.org/abs/2405.04083} {\path{arXiv:2405.04083}}.

\bibitem{PSAprep}
Mihai Prunescu and Lorenzo Sauras-Altuzarra.
\newblock On the representation of number-theoretic functions by arithmetic terms, 2024.
\newblock \href {https://arxiv.org/abs/2407.12928} {\path{arXiv:2407.12928}}.

\bibitem{PS2024}
Mihai Prunescu and Joseph Shunia.
\newblock Arithmetic-term representations for the greatest common divisor, 2024.
\newblock URL: \url{https://arxiv.org/abs/2411.06430}, \href {https://arxiv.org/abs/2411.06430} {\path{arXiv:2411.06430}}.

\bibitem{Roth1955}
Klaus~F. Roth.
\newblock Rational approximations to algebraic numbers.
\newblock {\em Mathematika}, 2:1--20, 1955.
\newblock \href {https://doi.org/10.1112/S0025579300000644} {\path{doi:10.1112/S0025579300000644}}.

\bibitem{Sta2012}
Richard~P. Stanley.
\newblock {\em Enumerative combinatorics. {Vol}. 1.}, volume~49 of {\em Camb. Stud. Adv. Math.}
\newblock Cambridge: Cambridge University Press, 2nd ed. edition, 2012.
\newblock URL: \url{www.cambridge.org/de/knowledge/isbn/item6832283/?site_locale=de_DE}.

\end{thebibliography}

\end{document}